\documentclass{amsart}
  
  \usepackage{amsmath,amssymb,amsthm}
 \usepackage{graphicx}
  \usepackage{xspace}
  \usepackage[all]{xy}
  \usepackage{pinlabel}
  \usepackage{enumerate}
  \usepackage{color}
  \usepackage{faktor}
  \usepackage{enumitem}

  \usepackage{hyperref}

  \newcommand{\calC}{\mathcal{C}}

  \newcommand{\calG}{\mathcal{G}}

  \newcommand{\calS}{\mathcal{S}}

  \newcommand{\CC}{\mathbb{C}}

  \newcommand{\RR}{\mathbb{R}}

  \newcommand{\ZZ}{\mathbb{Z}}

  \newtheorem{theorem}{Theorem}[section]
  \newtheorem{proposition}[theorem]{Proposition}
  \newtheorem{corollary}[theorem]{Corollary}
  \newtheorem{lemma}[theorem]{Lemma}

  \theoremstyle{definition}
  
  \newtheorem{claim}[theorem]{Claim}
  
  \newtheorem*{claim*}{Claim}

  \newtheorem*{question*}{Question}
  \newtheorem*{answer*}{Answer}
  \newtheorem*{application*}{Application}

  \theoremstyle{remark}
  
  \newtheorem*{remark*}{Remark}

  \DeclareMathOperator{\Mod}{Mod}
  \DeclareMathOperator{\Ext}{Ext}
  \DeclareMathOperator{\Hyp}{Hyp}

  \DeclareMathOperator{\I}{i}
  
  \DeclareMathOperator{\Area}{Area}

  \newcommand{\T}{\ensuremath{\mathcal{T}}\xspace} 
  \newcommand{\MF}{\ensuremath{\mathcal{MF}}\xspace} 
  \newcommand{\PMF}{\ensuremath{\mathcal{PMF}}\xspace} 
  \newcommand{\ML}{\ensuremath{\mathcal{ML}}\xspace} 
   
  \newcommand{\PML}{\ensuremath{\mathcal{PML}}\xspace}  
 \newcommand{\PC}{\ensuremath{\mathcal{PC}}\xspace}  
    
  \newcommand{\Teich}{{Teichm\"uller }} 
  \newcommand{\param}{{\mathchoice{\mkern1mu\mbox{\raise2.2pt\hbox{$
  \centerdot$}}
  \mkern1mu}{\mkern1mu\mbox{\raise2.2pt\hbox{$\centerdot$}}\mkern1mu}{
  \mkern1.5mu\centerdot\mkern1.5mu}{\mkern1.5mu\centerdot\mkern1.5mu}}}

  \renewcommand{\setminus}{{\smallsetminus}}

  \newcommand{\from}{\colon\thinspace} 

  \newcommand{\dT}{{d_T}}

  \newcommand{\balpha}{{\overline \alpha}}

  \numberwithin{equation}{section}

\begin{document}


  \title[ Teichm\"uller disks with small limit sets in $\mathcal PMF$]   
  {Teichm\"uller disks with small limit set in $\mathcal PMF$}
  

  \author{Anna Lenzhen}
  \address{IRMAR, 
Universit\'e de Rennes 1,
Campus de Beaulieu,
35042 Rennes Cedex, France}
\email{anna.lenzhen@univ-rennes1.fr}

  \date{\today}

\begin{abstract} 
 We study limit sets of  \Teich disks in the Thurston boundary of \Teich space of a closed surface $S$ of genus at least 2.  It is well known that almost every \Teich geodesic ray converges to a point on the boundary. We show that unlike \Teich rays, \Teich disks with smallest possible limit sets are extremely rare.
\end{abstract}
  
  \maketitle
\section{Introduction}  

For a closed oriented surface of genus at least two $S$, the Teichm\"uller space $T(S)$ is topologically an open ball. Thurston compactified it by adding the sphere of projective measured foliations $\PMF$. Equipped with the \Teich metric $d_T$, it  is a complete totally geodesic metric space which is not hyperbolic in the sens of Gromov, but which has many properties of a hyperbolic space. Its group of isometries is the mapping class group $MCG(S)$ \cite{royden:AT}, whose  action  extends continuously to the Thurston boundary $\PMF(S)$.


A Teichm\"uller disk is an isometric embedding of  Poincar\'e disk to $(T(S), d_T)$. A natural question is what its limit set in $\PMF(S)$ can  look like.
Clearly,  the limit set of a disk contains that of every geodesic ray in the disk. A  generic geodesic ray converges to a unique point on the boundary (see \cite{kerckhoff:EBF} and \cite{masur:CG}), on the other hand there are \Teich rays that do not converge, whose limit set is a circle (see \cite{rafi:LSII}) or a simplex of dimension $d= g-1$ (see \cite{lenzhen:DDL}). Hence we do not expect the disk limit set to be something particularly nice unless we put extra conditions on the disk itself.  The goal is to understand what is the smallest possible limit set a \Teich disk can have, how likely it is for a \Teich disk to have this limit set and how the limit set looks in general.

Let $X\in T(S)$ and $q$ a holomorphic quadratic differential on $X$. Let $D(X,q)\subset T(S)$ be the \Teich disk defined by the couple $(X,q)$.
We will denote $\Lambda(X,q)$ the limit set of  $D(X,q)$ in  $\PMF(S)$. 

Define $C(X,q)$ to be the set of projective classes of the vertical foliation of $e^{2i\theta}q$ for $\theta\in \RR/\pi \ZZ$. Topologically $C(X,q)$ is a circle in $\PMF(S)$. 
We first observe that
 \begin{lemma} \label{contains}$\Lambda(X,q)$  contains the set $C(X,q)$. 
 \end{lemma}

 We would like to know if there are \Teich disks with the  limit set satisfying
 $$\Lambda(X,q)=C(X,q)$$
 and if yes, how common this property is.   
 
 Our main result is that there are very few such disks.
  \begin{theorem}\label{small} For every $g\geq 2$, up to the action of the mapping class group,  there are at most finitely many \Teich disks in $T(S_g)$  whose limit set is a circle. 
 \end{theorem}
 We should mention that the set of these special \Teich disks is not empty, we show it by  giving some examples in the last section of the paper.

So, unless the flat surface $(X,q)$ is very special, the  set $\Lambda(X,q)$ properly contains $C(X,q)$. In particular, for the \Teich disk of $(X,q)$ to have the smallest possible limit set, it is not enough that $(X,q)$ is a Veech surface. This fact was pointed out by C. Leininger and A. Kent already in  \cite{leininger:GV}. On the other hand, we will see later that for $\Lambda(X,q)$ to be a circle, $(X,q)$ has to be tiled by parallegrams, i.e. the $SL(2,\RR)$--orbit  of $(X,q)$ has to contain a square-tiled surface, plus have  some additional property.

 We can nevertheless say something about limit sets of \Teich disks of Veech surfaces in general. Recall that any measured foliation $(F,\mu)$ defines a simplex in $\PMF$ whose interior consists of all projective measured foliations topologically equivalent to $(F,\mu)$. We will denote this simplex by $\Delta(F,\mu)$ and its interior by $\Delta^{\mathrm{o}}(F,\mu)$. Now fix $(X,q)$ and denote
 $$M(X,q)=C(X,q)\cup\left(\underset{\theta\in \RR/\pi\ZZ}{\cup}\Delta^{\mathrm{o}}(F_\theta,\mu_\theta)\right)$$
 where $(F_\theta,\mu_\theta)$ is the vertical measured foliation of the quadratic differential $e^{2i\theta}q$ on $X$.

 \begin{theorem} \label{Veech} Let $(X,q)$ be a Veech surface.
 The limit set of the \Teich disk defined by $(X,q)$ satisfies
 $$\Lambda(X,q)\subset M(X,q).$$
 \end{theorem}
The result stated above says that any accumulation point of $D(X,q)$ is topologically equivalent to a point in $C(X,q)$ and in the case of a not uniquely ergodic measured foliation, all the components have positive weight. Note that if $(X,q)$ is not Veech, the statement above does not hold in general, since it does not generally hold even for a \Teich ray. Indeed there are examples (\cite{lenzhen:DDL,rafi:LS}) of rays where the limit set is the entire simplex, boundary included. 

\paragraph{\textbf{Acknowledgements}} The author thanks Duc-Manh Nguyen, Bram Petri and Juan Souto for helpful conversations and encouragement.  The author gratefully acknowledges support from ANR grant MoDiff.
 
 \section{Preliminaries} In this section we recall some relevant definitions and facts from \Teich theory. For more details we refer the reader to \cite{abikoff:TT}, \cite{farc:MCG}, \cite{it:Teich},etc.
 \subsection{\Teich space}

Let $S$ be a closed surface of genus $g\geq 2$. 
 The \Teich space $T(S)$ is the space of marked complex structures on $S$
  up to isotopy.  By the uniformization theorem, the space $T(S)$ can be viewed as a space of finite area, complete, hyperbolic
  metrics on $S$ up to isotopy.
  
  We will be working with the  \Teich metric on $T(S)$. 
  Given  $X, Y \in T(S)$, the \emph{\Teich distance}
  between them is defined to be \[ \dT(X,Y) = \frac{1}{2} \inf_f \log
  K(f),\] where $f \from X \to Y$ is a $K(f)$--quasi-conformal homeomorphism
  preserving the marking. (See \cite{gardiner:QT} and \cite{hubbard:TT} for
  background information.) Geodesics in this metric are called \Teich geodesics, and we now recall a
  natural way to describe them.

 Let $X\in T(S)$ and $q=q(z)dz^2$ be a quadratic differential on $X$. There
 exists a  \emph{natural parameter} $\zeta=\xi+i\eta$, 
which is defined away from its singularities as
$$\zeta(w)=\int_{z_0}^{w}\sqrt{q(z)}\, dz.$$
In these coordinates, we have $q=d\zeta^2$. 
The lines $\xi=const$ with transverse measure $|d\xi|$ define the 
\textit{vertical} measured foliation, associated to $q$. Similarly, 
the \textit{horizontal} measured foliation is defined by $\eta=const$  and 
$|d\eta|$.  The transverse measure of an arc $\alpha$ with respect to 
$|d\xi|$, denoted by $h_q(\alpha)$, is called the \textit{horizontal length} of 
$\alpha$. Similarly, the \textit{vertical length} $v_q(\alpha)$ is the measure 
of $\alpha$ with respect to $|d\eta|$.

Given a marked Riemann surface
$X_0$ and a quadratic differential $q$ on $X_0$, we can obtain a
$1$--parameter family of quadratic differentials $q_t$ from $q$ 
so that, for $t\in \mathbb R$, if $\zeta=\xi+i\eta$ are natural coordinates for $q$, 
then $\zeta_t=e^{t} \xi+i e^{-t} \eta$ are natural coordinates for $q_t$. 
Let  $X(q,t)$ be the conformal structure associated to $q_t$. Then 
$\calG:\mathbb R\to T(S)$ which sends $t$ to $X(q,t)$, is a \Teich geodesic. Most of the time we will  work with the part
$\calG:\mathbb R_{+}\to T(S)$. We will refer to it as \textit{\Teich geodesic ray} $X_t$ based at $X_0$ and defined by $q$.

 A quadratic differential is called \textit{Strebel differential} if it decomposes the surface into 
 cylinders swept out by vertical trajectories. If $\gamma_i,\, i=1,\ldots,k$ are the core curves of the cylinders,  
 we simply denote the vertical foliation $(F,\mu)$ of $q$ by  $(F,\mu)=\sum_{i=1}^k h_i\gamma_i$.
 Here the weight $h_i$ is the height of the cylinder about $\gamma_i$ in the flat metric defined by $q$. We we refer any two such curves $\gamma_i$ and $\gamma_j$ as
 \textit{parallel curves.}


 \subsection{Lengths, intersection numbers, geodesic currents and Thurston boundary } 
 \subsection*{Lengths} There are three notions of length of a curve that will be useful to us.
 By a \textit{curve} on $S$ we  mean a free homotopy class of non-trivial non-peripheral  closed 
 curve on $S$. We will denote $\calS=\calS(S)$ the set of simple curves on $S$. 
 
 Let $X\in T(S)$.  Every curve $\alpha$ has a unique  geodesic representative in the hyperbolic metric of $X$, and we will
 denote its length by $\Hyp_X(\alpha)$. We call it the \textit{hyperbolic length }of $\alpha$ on $X$. If $\alpha$ is a set of curves, then $\Hyp_X(\alpha)$ is  the sum of the lengths 
 of the geodesic representatives of curves in $\alpha$. 
 By a \textit{short marking } of $X$ we mean a collection $\balpha=\{\alpha_1,\alpha_2,\ldots,\alpha_{6g-6}\}$ of simple curves so that 
$\balpha$ fills $X$ (that is, every simple curve has non-zero intersection number with some $\alpha_i$) and $\Hyp_X(\balpha)$ is smallest possible. 
 \begin{lemma}[Collar Lemma{\cite{buser:GSC}}]

    For any hyperbolic metric $X$ on $S$, if $\alpha$ is a  geodesic curve
    with $\Hyp_X(\alpha) = \epsilon$, then the regular neighborhood $U(\alpha)$
    of $\alpha$ with width $\omega(\epsilon)$ where
     \begin{equation}\label{width collar}
      \omega(\epsilon) = \sinh^{-1} \left(
    \frac{1}{\sinh(\epsilon)} \right),
    \end{equation}
     is an embedded annulus. 
  \end{lemma}
 A holomorphic quadratic differential $q$ on $X$ defines a locally flat metric on $X$ with singularities at zeros of $q$. 
 A \textit{saddle connection} is a geodesic segment with endpoints at zeros of $q$ and whose interior is disjoint from the zero set. Any curve $\alpha$ 
 always has a $q$--geodesic representative. It  might not be unique; there can be a family of parallel copies of $q$--geodesics foliating a flat cylinder. In case there is no flat cylinder, the $q$-geodesic is unique and is a concatenation of finitely many saddle connections. 
 We denote $\ell_{(X,q)}(\alpha)$ the length of a $q$--geodesic or simply $\ell_X(\alpha)$ if there is no ambiguity. We call it the flat length of $\alpha$ on $X$. 
 
 Let $Y$ be a subsurface of $S$. The homotopy class of $Y$ has a representative with $q$-geodesic boundary that is disjoint from the interior
 of the (if any) flat cylinder of every boundary component of $Y$. We will call it \textit{$q$--representative}  of $Y$. 
 
 The \textit{extremal length} of a simple curve $\alpha$ on $X$ is defined by
 $$ \Ext_X(\alpha) =\sup_{\rho}\frac{\ell_\rho(\alpha)^2}{Area_\rho(X)}$$
 where $\rho$ is any metric in the conformal class $X$, and $\ell_\rho(\alpha)$ is the  greatest lower bound of length of curves in the homotopy class of $\alpha$
 in the metric $\rho$. Extremal length can also be computed as the reciprocal of the modulus of the largest cylinder with core curve $\alpha$:
 \begin{equation}\label{Maxmod}
 \Ext_X(\alpha)=\frac{1}{\sup_{C_\alpha}\Mod_X(C_\alpha)}.
 \end{equation}
 
 Maskit established the following comparison result between hyperbolic and extremal lengths.
\begin{theorem}[\cite{maskit:HE}] \label{maskit} Let $X$ be a hyperbolic surface of finite type and $\alpha\subset X$ a curve. Then
 \begin{equation}\nonumber
 \frac{1}{\pi}\leq\frac{\Ext_X(\alpha)}{\Hyp_X(\alpha)}\leq \frac 1 2 e^{\Hyp_X(\alpha)}.
 \end{equation}
Hence, considering both lengths as functions on the Teichm\"uller space, the two lengths of $\alpha$ go to 0 together. Moreover,
 $$\lim\frac{\Ext_X(\alpha)}{\Hyp_X(\alpha)}\to \frac 1 \pi$$ if one of the lengths goes to $0$.
\end{theorem}
\subsection*{Measured foliations} 
We denote by $\MF(S)$ the space of measured foliations on $S$, and by $\PMF(S)$ the space of projective measured foliations, that is the space 
of measured foliations up to scaling. We refer to   \cite[expos\'e 8]{flp:TTs} 
for a detailed discussion of the facts stated here. Any simple curve $\alpha$
on $S$ determines a measured foliation whose non-singular leaves are homotopic to $\alpha$. The set $\RR_+\times \calS$ is dense in $\MF(S)$
and the function 
$$\I:\calS\times\calS \to \RR_+$$
extends via this inclusion to a unique continuous homogeneous function,
$$\I:\MF(S)\times\MF(S) \to \RR_+.$$
Also, Kerckhoff (see \cite{kerckhoff:AG}) proved that the extremal length function $\Ext: T(S)\times \calS\to \mathbb R_+$ has a unique continuous, square-homogeneous in the second factor 
extension to $\MF(S)$, 
$$\Ext: T(S)\times \MF(S)\to \mathbb R_+.$$
  \subsection*{Measured laminations} There is a closely related theory of measured laminations on $S$. For a fixed hyperbolic metric on $S$, a measured geodesic lamination is a closed subset of $S$ foliated by complete, simple geodesics, together with a Radon invariant measure on transverse arcs. The space of measured laminations is denoted $\ML(S)$. Any closed geodesic can be seen as  a measured lamination, and $\RR_+\times \calS$ is dense in $\ML(S)$. 
  Also, the function $\I:\calS\times\calS \to \RR_+$ extends naturally to a continuous homogeneous function $\I:\ML(S)\times\ML(S) \to \RR_+.$ Moreover,
  the function $\Hyp:T(S)\times \calS\to \RR_+$ extends to a continuous and homogeneous in the second factor function $\Hyp: T(S)\times \ML(S)\to \RR_+.$
  
  We will mostly work with measured foliations, but we need to mention that there is an identification between $\ML(S)$ and $\MF(S)$, see (\cite{levitt:FL}) for details. In particular,
  this identification respects the intersection number functions.
  
  \subsection*{Geodesic currents and Thurston boundary}
Both $T(S)$ and $\ML(S)$ can be embedded in the space of geodesic currents $\calC(S)$: the space of Radon $\pi_1(S)$-- invariant measures on the space of geodesics in 
the universal cover of $S$, see \cite{bonahon:TC} for detailed expositions of the following facts. There is an intersection number function
$$\iota:\calC(S)\times \calC(S)\to \mathbb R$$ that is continuous and homogeneous in both coordinates and that agrees with the intersection number function on $\ML(S)\times\ML(S)$ and the hyperbolic length function on $T(S)\times \ML(S)$. That is, for any $\mu,\nu\in \ML(S)$ and $X\in T(S)$,
we have $\iota(\mu,\nu)=i(\mu,\nu)$ and $\iota(X,\mu)=\Hyp_X(\mu)$.

The quotient $\PC(S)$ of $\calC(S)-\{0\}$ by the action of $\RR_{>0}$ is compact, and the embedding of $T(S)$ and $\ML(S)$ descends to  the embeddings of $T(S)$
and $\PML(S)$. The closure of $T(S)$ in $\PC(S)$ is the Thurston compactification of $T(S)$. The Thurston boundary of $T(S)$ is identified with both $\PML(S)$ and $\PMF(S)$. We will rather use the  $\PMF(S)$ point of view.

 A sequence of points $X_n\in T(S)$ converges to a projective measured
   foliation $[(F,\nu)]$ if and only if for any two curves $\gamma_1,\gamma_2$ on $S$ we have
 \[
 \lim_{n\to\infty}\frac{\Hyp_{X_n}(\gamma_1)}{\Hyp_{X_n}(\gamma_2)}=\frac{\I(\gamma_1,(F,\nu))}{\I(\gamma_2,(F,\nu))}.
 \]
  The following fundamental fact is  due to H. Masur (\cite{masur:CG} and \cite{masur:TB}).
 \begin{theorem}\label{convrays}
 Let $X\in T(S)$ and $q$ a holomorphic quadratic differential on $X$. Let $X_t$ be the \Teich ray based at $X$ and determined by $q$. Let $(F,\mu)$ be the vertical measured foliation of $q$. Then
 \begin{enumerate}
 \item If $q$ is Strebel and $(F,\mu)=\sum_{i=1}^k h_i\gamma_i$, then the ray $X_t$ converges  in $\PMF$ to the baricenter $[\sum_{i=1}^k \gamma_i]$.
 \item If $(F,\mu)$ is uniquely ergodic, then the ray $X_t$ converges  in $\PMF$ to the projective class of $(F,\mu)$, $X_t\to[(F,\mu)]$.
 \end{enumerate}
 \end{theorem}
 

 \subsection{Teichm\"uller disks}
 Let $X\in T(S)$ and $q$ a holomorphic quadratic differential on $X$. We denote $(X,q)$ the flat surface with the singular flat metric defined by $q$. 
Denote $$D=\{re^{i\theta} ,r\in [0,1), \theta\in \RR/2\pi \ZZ\}$$ be the Poincar\'e disk model of the hyperbolic space and let 
\begin{equation}\label{embedding} \phi:D\to T(S),\,\, \phi(re^{i\theta})\mapsto X(e^{i\theta}q,\log\frac{1+r}{1-r}).\end{equation}
Then $\phi$ is an isometric embedding of $D$ to $T(S)$. We denote the image $\phi(D)$ by $D(X,q)$ and we call it the \Teich disk generated by $(X,q)$.\\

 The group $SL(2,\RR)$ acts on flat surfaces as follows. Let $A\in SL(2,\RR)$. The new surface $A\cdot (X,q)$ is obtained by post-composing 
 the coordinate functions with $A$ acting $\RR$-linearly on $\CC$. 
 The set  $D(X,q)$ can also be seen as the projection to $T(S)$ of the $SL(2,\RR)$--
 orbit of $(X,q)$,  see for example \cite{schmith:TD} for details.

 Let $(F_\theta,\mu_\theta)$ be the vertical measured foliation of $e^{2i\theta} q$. We call $\theta$ 
 \textit{completely periodic} 
 if the differential $e^{2i\theta} q$ is Strebel and the moduli of the flat cylinders of $(F_\theta,\mu_\theta)$ are rationally related. We refer to  $\theta$ as 
  \textit{uniquely ergodic} if the measured foliation $(F_\theta,\mu_\theta)$ admits a unique up to scaling measure.
  The following theorem that is crutial for us was proved by Kerckhoff, Masur and Smillie.
  \begin{theorem}[\cite{kerckhoff:EBF}]\label{ue}Let $(X,q)$ be a flat surface. Almost every direction $\theta$ on $(X,q)$ is uniquely ergodic.
  \end{theorem}

\subsection*{Veech surfaces}
Let $(X,q)$ be a flat surface.  The Veech group of $(X,q)$ is the set of linear parts of affine (in charts) automorphisms of $(X,q)$. It is a subgroup of $PSL(2,\RR)$.
If the Veech group of $(X,q)$ is a lattice, $(X,q)$ is called a \textit{Veech surface}. These surfaces have many very nice properties. First they satisfy Veech dichotomy:
\begin{theorem}[\cite{veech:TC}] Let $(X,q)$ be a Veech surface. Then any direction $\theta$ on $(X,q)$
is completely periodic or uniquely ergodic. 
\end{theorem}
Moreover being a Veech surface is  equivalent to following properties.
\begin{theorem}[\cite{smillie:SW}] \label{SW}The following is equivalent for a flat surface $(X,q)$.
\begin{enumerate}
\item $(X,q)$ is a Veech surface;
\item  For each direction $\theta$ on $(X,q)$ in which there is a saddle connection, $\theta$ is completely periodic;
\item  There is $s>0$ such that, for each direction $\theta$ on $(X,q)$ in which there is a saddle connection, $\theta$ is periodic, and the ratio of lengths of any two saddle connections in direction $\theta$ is at most $s$.
\end{enumerate}
\end{theorem}
The most interesting example of Veech surfaces for us will be \textit{square-tiled surfaces}, that is, flat surfaces that can be tiled by squares of equal size. These surfaces are Veech by  a theorem of Gutkin-Judge (see \cite{gutkinjudge:AR}).
\section{Some  properties of the limit set of a Teichm\"uller disk. }
In this section we will establish several general facts about limit sets of \Teich disks.
 We start with the following technical statement.
 \begin{lemma}\label{extcyl}Let $X\in T(S)$ and $q$ a quadratic differential on $X$. Let $(F,\mu)$ be the vertical measured foliation of $q$ and suppose that it has a cylinder component with core curve $\gamma$. Let $X_t$ be the \Teich ray based at $X$ and defined by $q$. Then
 along $X_t$, the hyperbolic and extremal lengths of $\gamma$ satisfy
\begin{enumerate} 
\item $\Ext_{X_t}(\gamma)\cdot e^{2t}\Mod_{X}(\gamma)\to 1;$
\item $\Hyp_{X_t}(\gamma)\cdot e^{2t}\Mod_{X}(\gamma)\to \pi.$
\end{enumerate}
 \end{lemma}
\begin{proof} We will prove the statement about the extremal length. The second statement follows from the first one by Maskit's Theorem \ref{maskit}.

Recall that the extremal length of a simple curve is bounded above by  the reciprocal of the modulus of any cylinder about the curve. Along the Teichm\"uller ray $X_t$ the modulus of the flat cylinder about $\gamma$ is $\Mod_{X_t}(\gamma)=\frac{e^th}{e^{-t}\ell}=e^{2t}\Mod_{X}(\gamma)$, where $h$ is the height of the flat cylinder on $X$ and $\ell$ is the $q$-length of $\gamma$ . 

We then have the upper bound for the extremal length
$$\Ext_{X_t}(\gamma)\leq\frac{1}{\Mod_{X_t}(\gamma)}=\frac{1}{e^{2t}\Mod_{X}(\gamma)}$$
which gives
$$\Ext_{X_t}(\gamma)\cdot e^{2t}\Mod_{X}(\gamma)\leq 1.$$
To get a  lower bound, we can consider a metric $\rho_t|dz|$ on $X_t$ with $\rho_t$ equal to 1 on the $\ell_{X_t}(\gamma)$- neighborhood  of the maximal flat cylinder about $\gamma$ and zero elsewhere. Then the length of any curve homotopic to $\gamma$ in this metric is at least $\ell_{X_t}(\gamma)$ and the area of the surface is $h\ell+C\ell^2_{X_t}(\gamma)$ for some universally bounded $C$. Hence
$$\Ext_{X_t}(\gamma)\geq \frac{\ell^2_{X_t}(\gamma)}{h\ell+C\ell^2_{X_t}(\gamma)}= \frac{\ell_{X}(\gamma)}{e^{2t}h+C\ell_{X}(\gamma)}=\frac{1}{e^{2t}\Mod_X(\gamma)+C}$$ and so 
$$\Ext_{X_t}(\gamma)\cdot e^{2t}\Mod_X(\gamma)\geq \frac{e^{2t }\Mod_X(\gamma)}{e^{2t}\Mod_X(\gamma)+C}\to 1.$$

\end{proof}

Let us now  prove the Lemma \ref{contains} stated in the introduction. 
Recall that it says that the limit set of a \Teich disk $D(X,q)$ always contains $C(X,q)$.
 \begin{proof}[Proof of Lemma \ref{contains}] 
  Recall that the disk $D(X,q)$ is  a union of geodesic rays based at $X$ and defined by the
 quadratic differentials $e^{2i\theta}q$, $\theta\in \RR/\pi \ZZ$. In particular $\Lambda(X,q)$ contains every accumulation point of every \Teich ray.  Hence by Theorem \ref{convrays}, $\Lambda(X,q)$ contains $[F_\theta,\mu_\theta]$ for all $\theta$ ergodic. Moreover,  Theorem \ref{ue}  implies that almost every direction $\theta$ is ergodic. This means that almost every point of $\calC(X,q)$ is in $\Lambda(X,q)$. On the other hand,  $\Lambda(X,q)$ is a closed subset of $\PMF$, so it  contains all of $\calC(X,q)$.
 
 \end{proof}
 It is well known that any accumulation point in $\PMF$ of a \Teich geodesic ray has  zero intersection with the vertical foliation of the quadratic differential that generates $X_t$. The following lemma is a generalisation of this fact.
\begin{lemma}\label{intersection_zero}
Let $X\in T(S)$ and let $X_n\in T(S)$ be a sequence converging to $[(F,\mu)]\in\PMF$. Let $(G_n,\nu_n)$ be the vertical foliation of the quadratic differential on $X$ that determines the geodesic segment 
$[X,X_n]$, and
suppose $(G_n,\nu_n)\underset{n\to \infty}{\to} (G,\nu)$. Then $\I\left((F,\mu),(G,\nu)\right)=0$.
\end{lemma}
\begin{proof} To simplify the presentation, we write $\mu$ instead of $(F,\mu)$, $\nu$ instead of $(G,\nu)$ and $\nu_n$ for $(G_n,\nu_n)$. 
Let $\lambda_n$ be the Liouville geodesic current of $X_n$. Let $s_n\to 0$ be such that 
$$s_n\cdot \lambda_n\to \mu.$$
The intersection number is continuous and homogeneous on $\calC(S)\times \calC(S)$, so we have
\begin{equation} \underset{n\to +\infty}{\lim}\I(s_n\lambda_n,\nu_n)= \underset{n\to +\infty}{\lim}s_n\I(\lambda_n,\nu_n)=\I(\mu,\nu).
\end{equation}
To conclude that $\I(\mu,\nu)=0$ it suffices to show that the sequence $\I(\lambda_n,\nu_n)$ is bounded. That is, to show that the hyperbolic length $\Hyp_{X_n}(\nu_n)$ is bounded. 

We argue as follows. Continuity of the extremal length function implies that $\Ext_X(\nu)$ is finite and so
 $\Ext_X(\nu_n)$ is bounded.  Also  along the \Teich geodesic segment $[X,X_n]$ the extremal length of $\nu_n$
 is decreasing, so we can conclude that $\Ext_{X_n}(\nu_n)$ is also bounded.  Now by definition of extremal length  the hyperbolic length is smaller than the square root of extremal length up to a bounded multiplicative error. We have that 
 $$\I(\lambda_n,\nu_n)=\Hyp_{X_n}(\nu_n)\leq C \sqrt{\Ext_{X_n}(\nu_n)}$$
 is bounded. Hence $\I(\mu,\nu)=0$ and this finishes the proof.
\end{proof}
For a subset $E\subset \PMF$, denote 
$$Z(E)=\underset{[(F,\mu)]\in E}{\cup}\{[(G,\nu)]\in\PMF, \I((G,\nu), (F,\mu))=0\}.$$
We have the following immediate consequence of the lemma.
\begin{corollary}  The limit set of a \Teich disk $D(X,q)$ is contained in $Z(C(X,q))$.
\end{corollary}

We now prove a statement that will allow us to eliminate from the race  \Teich disks of flat surfaces with foliations that have both minimal and cylinder components. 
\begin{proposition} \label{periodic}
 Let $X\in T(S)$, $q$ a quadratic differential on $X$ and $X_t$ the \Teich ray based at $X$ and defined by $q$. Let $(F,\mu)$ be the vertical measured foliation of $q$ and suppose it has a cylinder component. 
 If $X_t$ converges in $\PMF$ to $[(F,\mu)]$,  then $(F,\mu)$ has no minimal components. 
\end{proposition}
\begin{proof}
 We will show that $X_t$ cannot converge to the projectivized measured foliation $[(F,\mu)]$ if $(F,\mu)$ has both a flat cylinder and a minimal component.
So suppose that there are both a minimal component and a flat cylinder,  both of positive area.  

Let $Y\subset X$  be the minimal component. Rescale $q$ so that $\Area_{q}(Y)=1$. As $t\to \infty$, the hyperbolic, flat and extremal lengths of every  boundary component of $Y$ tend to $0$. 

Let $\epsilon_0>0$ be smaller than the Margulis constant and fix a small $\epsilon<\epsilon_0$. Let $N_\epsilon$ be such that for all $ t>N_\epsilon$ and for any boundary component $\gamma$ of $Y$, we have 
$$\max\{\Ext_t(\gamma),\ell_t(\gamma), \Hyp_t(\gamma)\}<\epsilon.$$
\begin{claim} \label{sequence}There are constants $C_1,C_2 >0$, a sequence  $t_n\to +\infty$ and a sequence of distinct simple closed curves $\alpha_n\subset Y$ 
such that 
\begin{enumerate}\label{short}
\item $v_{t_n}(\alpha_n)\geq C_1$.
\item $ \Hyp_{t_n}(\alpha_n)\leq C_2$.
\end{enumerate}
\end{claim}
\begin{proof} 

Let $t>N_\epsilon$. Suppose first that  the $q_t$ representative of $Y$ contains a flat cylinder $\calC $ with $$\Area_{q_t}(\calC)\geq \frac{1}{3g-3}.$$ and modulus $\Mod_t(\calC)>1$. 
Let $\alpha$ be its core curve. 
The vertical component of $\alpha$ could be zero, and this curve would not be useful to us.  But  any such curve is pinched along the Teichm\"uller geodesic ray in the opposite direction, hence the number of such curves is at most $3g-3$. Taking sufficiently large $N_\epsilon$ guarantees that $\alpha$ has positive vertical component. 
Then  there is $t_\alpha\in \mathbb R$ so that  $\Mod_{t_\alpha}(\calC)=1$ and the curve $\alpha$ is vertical, that is $v_{t_\alpha}(\alpha)\geq h_{t_\alpha}(\alpha)$.

Now the extremal length of $\alpha$ on $X_{t_\alpha}$ is bounded below by $1$:
$$\Ext_{t_\alpha}(\alpha)\leq \frac{1}{\Mod_{t_\alpha}(\calC)}=1$$
and so the hyperbolic length of $\alpha$ is also bounded (follows from Theorem \ref{maskit}):
$$\Hyp_{t_\alpha}(\alpha)\leq \pi \Ext_{t_\alpha}(\alpha)\leq \pi.$$
The flat length of $\alpha$ satisfies $$\ell_{t_\alpha}(\alpha)=\sqrt {\Area_{q_{t_\alpha}}(\calC)}\geq \frac{1}{\sqrt{3g-3}}.$$ 
Also the inequality $v_{t_\alpha}(\alpha)\geq h_{t_\alpha}(\alpha)$ and the fact that $\ell_{t_\alpha}(\alpha)\leq h_{t_\alpha}(\alpha)+v_{t_\alpha}(\alpha)$ imply
$$v_{t_\alpha}(\alpha)\geq \frac 1 2 \ell_{t_\alpha}(\alpha)\geq \frac {1}{6g-6}.$$ 
 
Suppose now that at time $t$ there is no big area and big modulus flat cylinder in $Y$. Then
there is an $\epsilon_0$--thick subsurface $W\subset Y$ with $\Area_{q_t}(W)\geq \frac{1}{3g-3}$. Let $(\alpha_i)_{1\leq i\leq k}$ be a short marking of $W$.
It follows from Theorem 1 in \cite{rafi:TT} that there is a constant $C_2$ such that for any curve in the short marking we have
 $$\Hyp_t(\alpha_i),\ell_t(\alpha_i)\in [\frac{1}{C_2},C_2].$$ We 
want to show that for at least one of these curves the vertical component is comparable to the flat length. 

Suppose for some  $\delta>0$ and all $i$ we have
$$v_t(\alpha_i)\leq \delta\ell_t(\alpha_i).$$ 

Cut the $q_t$-representative of $W$ along  $q_t$-geodesic representatives of $\alpha_i$ to get a finite collection of polygons $(P_j)_{1\leq j\leq m}$ or polygons with holes. The number of these polygons is bounded by a number that depends on topology of $X$ only. The area of each polygon $P_j$ can be (generously) estimated by
$$\Area_{q_t}(P_j)\leq \sum_{i=1}^{k} v_t(\alpha_i)\sum_{i=1}^k h_t(\alpha_i)\leq \delta \left[\sum_{i=1}^k \ell_t(\alpha_i)\right]^2.$$
Adding the area of all the polygons we now get that the area of $W$ is bounded above by 
$$\Area_{q_t}(W)\leq \delta m \left[\sum_{i=1}^k\ell_t(\alpha_i)\right]^2$$ 
and this  implies that
$$\delta\geq\frac{\Area_{q_t}(W)}{m\left[\sum_{i=1}^k\ell_t(\alpha_i)\right]^2}.$$
The righthand side is bounded below by a positive universal constant, so choosing $C_1>0$ small enough guarantees that
for  some $i$, 
$$v_t(\alpha_i)\geq C_1.$$
Let $t_n\to +\infty$ be a sequence  of times and $(\alpha_{t_n})$ the sequence of curves constructed this way. Since the foliation $F$ is minimal in $Y$, every non-peripheral simple closed curve in $Y$ has positive horizontal component. This implies that eventually the hyperbolic length of every such curve goes to infinity. Hence 
we can assume the curves $(\alpha_{t_n})$ are all  distinct.
 \end{proof} 
Continuing with the proof of the Proposition, let $(t_n)$ and $(\alpha_n)$ be as in the Claim \ref{sequence}. Denote $(F^h,\mu^h)$ the horizontal foliation of $q$. Also denote $\gamma$ the core curve of the cylinder component of $(F,\mu).$\\
Consider the sequence of weighted curves $\frac{1}{v_{0}(\alpha_{t_n})}\alpha_{t_n}$. 
We have $$\lim_{n\to +\infty}\I(\frac{1}{v_{0}(\alpha_{t_n})}\alpha_{t_n},(F^h,\mu^h))=\lim_{n\to +\infty}\frac{v_0(\alpha_{t_n})}{v_0(\alpha_{t_n})}=1,$$ and
$$\lim_{n\to +\infty}\I(\frac{1}{v_{0}(\alpha_{t_n})}\alpha_{t_n},(F,\mu))=\lim_{n\to+\infty}\frac{h_{0}(\alpha_{t_n})}{v_{0}(\alpha_{t_n})}=0.$$
Note that it follows from \ref{short} and the fact that hyperbolic length cannot grow faster than exponentially that the hyperbolic lengths of the weighted curves $\Hyp_{X_0}(\frac{1}{v_{0}(\alpha_{t_n})}\alpha_{t_n})$ are also bounded. Since hyperbolic length function is proper, 
we can assume (up to passing to a subsequence) that $\frac{1}{v_{0}(\alpha_{t_n})}\alpha_{t_n}$ converges to some measured foliation $(G,\eta)$.
 Note that 
 $$\I((G,\eta),(F^h,\mu^h))=\lim_{n\to +\infty}\frac{v_0(\alpha_{t_n})}{v_0(\alpha_{t_n})}=1,$$ so $(G,\eta)$ is non-zero. Since the curves $\alpha_n$ are all contained in $Y$, the support of $\eta$ is also contained in $Y$.
 We have by continuity of the intersection number function that
 $$\I((G,\eta),(F,\mu))=\lim_{n\to+\infty}\frac{1}{v_{0}(\alpha_{t_n})}\I(\alpha_{t_n},(F,\mu))=\lim_{n\to+\infty}\frac{h_{0}(\alpha_{t_n})}{v_{0}(\alpha_{t_n})}=0,$$
 and minimality of $F$ restricted to $Y$ implies that $G$ is topologically the foliation $F|_Y$. In particular, any closed curve that intersects $Y$ has positive intersection with $(G,\eta)$.

Let $\beta$ be any closed curve such that $\I(\beta,\gamma)>0$ and $\I(\beta,(F|_Y,\mu|_Y))>0$.  Then for $n$ big enough, the hyperbolic 
length of $\beta$ restricted to $Y$ satisfies
 $$\Hyp_{X_{t_n}}(\beta, Y)\geq \omega(C_2)\I(\beta,\alpha_{t_n})\sim \omega(C_2)v_0(\alpha_{t_n})\I(\beta,(G,\eta))\geq \omega(C_2)C_1e^{t_n}\I(\beta,(G,\eta)).$$

Let  us estimate the contribution to the length of $\beta$  on $X_{t_n}$ from the standard collar around $\gamma$. From Lemma \ref{extcyl} we have
$$\Hyp_{X_{t_n}}(\gamma)\sim \frac{\pi}{e^{2t_n}\Mod_X(\gamma)} .$$
Hence the width of the collar about $\gamma$ grows linearly, $\omega(\Hyp_{t_n}(\gamma)) \sim 2t_n$. The twisting of $\beta$ about $\gamma$ does not change significantly with $n$, so we conclude that the  hyperbolic length of $\beta$ on $X_{t_n}$ restricted to the standard collar about $\gamma$ is
\begin{equation}
\Hyp_{X_{t_n}}(\beta, \calC(\gamma))\sim \I(\beta,\gamma)\cdot 2t_n. 
\end{equation}
We obtain that the contribution from $\gamma$ to the length of $\beta$ is negligible compared to the contribution of $Y$
\begin{equation}\label{contribution}
\underset{n\to \infty}{\lim}\frac{\Hyp_{X_{t_n}}(\beta,\calC(\gamma))}{\Hyp_{X_{t_n}}(\beta,Y)}= 0.
\end{equation}

Recall that we assumed that $X_t$ converges to $[(F,\mu)]$. Let us argue that  \ref{contribution} contradicts this assumption. 

Consider any  closed curve $\beta$ that intersects $Y$ and that intersects $\gamma$ twice. Let $\beta'$ be the (multi-)curve disjoint from $\gamma$ and obtained from $\beta$ as follows. Since $\beta$ intersects $\gamma$ twice, there are two arcs of $\beta$ on each side  of $\gamma$. Cut $\beta$ at the points of intersection with $\gamma$ and connect the arcs on the same side by an arc that goes twice around $\gamma$. We get a closed curve $\beta'$ that has self-intersections and possibly more than one component.  Now,
 $\beta$ and $\beta'$ intersect any simple closed curve disjoint from $\gamma$ equal number of times, and so we have $$\I(\beta, (F,\mu))=\I(\beta', (F,\mu))\,\,\, \text{ and }\,\, \I(\beta',\gamma)=0.$$ For some constant $C$ and any  $n$ big enough the hyperbolic lengths of $\beta$ and $\beta'$ satisfy
\begin{equation}
\Hyp_{X_{t_n}}(\beta')-C\leq\Hyp_{X_{t_n}}(\beta)\leq  \Hyp_{X_{t_n}}(\beta')+2\omega(\Hyp_{X_{t_n}}(\gamma)) +C,
\end{equation}
and this together with \ref{contribution} implies 
\begin{equation}
\frac{\Hyp_{X_{t_n}}(\beta')}{\Hyp_{X_{t_n}}(\beta)}\underset{n\to \infty}{\to} 1
\end{equation}
which implies that $\gamma$ has zero weight in the limiting foliation of $X_t$. 
We conclude that measured foliation $(F,\mu)$ cannot have minimal components, and this finishes the proof.
\end{proof}

\begin{proposition} \label{onthenose}Let $(X,q)$ be a flat surface.  Let $\theta$ be a completely periodic direction on $(X,q)$, and denote $(F_\theta,\mu_\theta)=\sum_{i=1}^k h_i\gamma_i$ the corresponding vertical measured foliation. 
Then we have 
\begin{enumerate}
\item If for any $i,j\in\{1,\ldots,k\}$, the function
$$Y\mapsto \frac{\Hyp_Y(\gamma_i)}{\Hyp_Y(\gamma_j)}$$ is bounded on $D(X,q)$, then 
for any  $[(F,\mu)]\in \Lambda(X,q)$ that satisfies $$\I((F,\mu),(F_\theta,\mu_\theta))=0$$ we have
$[(F,\mu)]\subset  \Delta^{\mathrm{o}}(F_\theta,\mu_\theta)$.
\item Moreover, $$\Lambda(X,q)\cap \Delta(F_\theta,\mu_\theta)=[(F_\theta,\mu_\theta)]$$ if and only if,  for any $i,j\in\{1,\ldots,k\}$, the function
$$Y\mapsto \frac{\Hyp_Y(\gamma_i)}{\Hyp_Y(\gamma_j)}$$
is constant on  $D(X,q)$. 
\end{enumerate}
\end{proposition}

\begin{proof} 
Suppose  that for any $i,j\in\{1,\ldots,k\}$, the function $\frac{\Hyp_Y(\gamma_i)}{\Hyp_Y(\gamma_j)}$ is bounded or constant on $D(X,q)$. 

If the ratio of  hyperbolic lengths is constant in $D(X,q)$, 
it must be equal to the ratio $\frac{\I(\gamma_i,(F_{\theta'},\mu_{\theta'}))}{\I(\gamma_j,(F_{\theta'},\mu_{\theta'}))}$ of the intersection numbers with any uniquely ergodic foliation $(F_{\theta'},\mu_{\theta'})$ on $(X,q)$. 
Hence it is equal to the ratio of the flat lengths of $\gamma_i$ and $\gamma_j$, that is
\begin{equation}\label{constantratio}
 \frac{\Hyp_{Y}(\gamma_i)}{\Hyp_Y(\gamma_j)}=\frac{\ell_X(\gamma_i)}{\ell_X(\gamma_j)} \,\,\,\,\,\forall Y\in D(X,q).
\end{equation}

 Let us argue that this implies equality $h_i=h_j$. 
Recall that $h_i$ is the height of the cylinder with core curve $\gamma_i$ on the translation surface $(X,q)$. Consider the \Teich ray $X_t=X(e^{2i\theta}q,t)$ based at $X$ defined by the quadratic differential $e^{2i\theta}q$. 
Since both $\gamma_i$ and $\gamma_j$ have flat cylinders, by Lemma \ref{extcyl}, as $t\to +\infty$, the ratio of their hyperbolic  lengths become asymptotic to the reciprocal of the ratio of the moduli of the corresponding cylinders, that is 
we have 
$$\frac{\Hyp_{X_t}(\gamma_i)}{\Hyp_{X_t}(\gamma_j)}=\underset{t\to \infty}{\lim}\frac{\Mod_{X_t}(\gamma_j)}{\Mod_{X_t}(\gamma_i)}=\underset{t\to \infty}{\lim}\frac{\ell_{X_t}(\gamma_i)h_{X_t}(\gamma_j)}{\ell_{X_t}(\gamma_j)h_{X_t}(\gamma_i)}.$$
Taking into the account the Equation \ref{constantratio} we conclude that $h_{X_t}(\gamma_i)=h_{X_t}(\gamma_j)$ and hence 
\begin{equation} \label{eh}
h_i=h_j\,\,\, \forall i,j\in\{1,\ldots,k\}.
\end{equation}

Let $X_n\subset D(X,q)$ be a sequence that converges in $\PMF$ to $[(F,\mu)]$ and such that $\I((F,\mu),(F_\theta,\mu_\theta))=0$. We need to argue that $(F,\mu)=\sum_{i=1}^k c_i\gamma_i$. Moreover, if the ratios $\frac{\Hyp(\gamma_i)}{\Hyp(\gamma_j)}$ are constant on $D(X,q)$ for all $i,j\in\{1,2,\ldots, k\}$, we want to show that, up to scaling,  $c_1=c_2=\ldots=c_k=1$. If the ratios are only bounded, we want to show that all the coefficients are positive.

 Let $T_{\bar\gamma}$ be a Dehn multi-twist about $(\gamma_1,\ldots,\gamma_k)$ that fixes $D(X,q)$. We choose a fundamental domain $D_{\bar\gamma}\subset D(X,q)$ for the action of  $<T_{\bar\gamma}>$ on $D(X,q)$ so that the boundary components are bi-infinite geodesics and so that $X\in D_{\bar\gamma}$. In one direction they are asymptotic and converge to $[\sum_{i}{\gamma_i}]$ in $\PMF$. In the opposite direction they are asymptotic to rays $X(e^{2i\theta_1}q,t)$ and $X(e^{2i\theta_2}q,t)$ based at $X$. Let $[\theta_1,\theta_2] \subset \RR/\pi\ZZ$ be the segment that does not contain $\theta$.

 Write $T_{\bar\gamma}=\Pi_{i=1}^k T_{\gamma_i}^{n_i}$, where $T_{\gamma_i}$ is the Dehn twist about $\gamma_i$ and  for any $(i,j)\in \{1, 2,\ldots,k\}^2$,
 \begin{equation}\label{twisting ratio}
 \frac{n_i}{n_j}=\frac{\Mod_X(\gamma_i)}{\Mod_X(\gamma_j)}.
 \end{equation}

Let $k_n\in \ZZ$ be such that 
$$Y_n:=T^{k_n}_{\bar\gamma}(X_n)\subset D_{\bar\gamma}.$$  Up to passing to a subsequence, $Y_n$ converges in $T(S)\cup\PMF$. There are several cases to consider.
\begin{enumerate}[label=\Roman*.,leftmargin=*] 
\item Suppose first that the hyperbolic length of $\gamma_i$ along $Y_n$ is bounded above and below. Then, again up to passing to a subsequence, $Y_n$ converges to some $Y\in D_{\bar\gamma}$.  Let $\alpha$ and $\beta$ be simple closed curves. If they are disjoint from all $\gamma_i$, then their length is not changed when we apply the Dehn twist and hence stays bounded along $X_n$. So assume both curves have positive intersection with $(F_\theta,\mu_\theta)$. Along the sequence $T^{-k_n}_{\bar\gamma}(Y)$ we have 
\begin{equation}\label{bl}
\underset{n\to\infty}{\lim}\frac{\Hyp_{T^{-k_n}_{\bar\gamma}(Y)}(\alpha)}{\Hyp_{T^{-k_n}_{\bar\gamma}(Y)}(\beta)}=
\underset{n\to\infty}{\lim}\frac{\sum_{j=1}^k n_j\Hyp_{Y}(\gamma_j)\I(\alpha,\gamma_j)}{\sum_{j=1}^k n_j\Hyp_{Y}(\gamma_j)\I(\beta,\gamma_j)}.
\end{equation}
If the lengths ratios are constant, from  \ref{twisting ratio}, \ref{constantratio} and \ref{eh} it follows that
\begin{equation} \label{equal}
n_j \Hyp_{Y}(\gamma_j)=n_1 \Hyp_{Y}(\gamma_1),
\end{equation}
and we obtain
\begin{equation}\label{limittwist}
\underset{n\to\infty}{\lim}\frac{\Hyp_{T^{-k_n}_{\bar\gamma}(Y)}(\alpha)}{\Hyp_{T^{-k_n}_{\bar\gamma}(Y)}(\beta)}=\frac{\sum_{j=1}^k \I(\alpha,\gamma_j)}{\sum_{j=1}^k \I(\beta,\gamma_j)}.
\end{equation}
We see that the sequence $T^{-k_n}_{\bar\gamma}(Y)$ converges to $[(F_\theta,\mu_\theta)]$, then same is true for the sequence $X_n$ since 
$$\underset{n\to +\infty}{\lim} d_T(X_n,T^{-k_n}_{\bar\gamma}(Y))=\underset{n\to +\infty}{\lim}d_T(Y_n,Y)=0.$$

 In the case of bounded length ratios,  we can argue from Equation \ref{bl} that $X_n$ converges in $\mathcal{PMF}$ to $[\sum_i c_i\gamma_i]$ such that $c_i>0$ by noticing 
 that the quantities $\frac{n_i\Hyp_{Y}(\gamma_i)}{n_j\Hyp_{Y}(\gamma_j)}$ are positive and bounded.

\item Now suppose that the hyperbolic length of $\gamma_i$ goes to $0$ along $Y_n$. Then $Y_n$ converges, like the two geodesic rays forming the boundary of  $D(X)$, to $[\sum_{j=1}^k \gamma_j]$.
 If $\sup_n |k_n|$ is bounded, then $X_n$ is asymptotic to two geodesic rays that converge to the same point, and hence also converges to  $[\sum_{j=1}^k \gamma_j]$.
Let us consider the case  when $\sup_n |k_n|\to +\infty$. 

On the hyperbolic surface $X_n$, consider the standard collar $U_{i}$ about $\gamma_i$, and let $\eta\subset \bar U_{i}$ be a geodesic arc  that crosses the collar and has endpoints on its boundary. 
Then the length of $\eta$ is, up to a bounded additive error, the width of the collar plus the number of times the arc twists about $\gamma_i$ times the length of $\gamma_i$, that is,
\begin{equation}\label{arclength}
|\Hyp_{X_n}(\eta)-(\omega(\Hyp_{X_n}(\gamma_i))+tw_{X_n}(\gamma_i,\eta)\cdot\Hyp_{X_n}(\gamma_i))|\leq C.
\end{equation}
If this $\eta$ is an arc of $\alpha$, then on $X_n$ its twisting about $\gamma_i$ is $|k_nn_i|$ up to an additive error that depends on $\alpha$ and $\gamma_i$. This gives
\begin{equation}\label{thintwist}
\Hyp_{X_n}(\alpha)=\sum_{i=1}^k \left(\omega(\Hyp_{X_n}(\gamma_i))+|k_nn_i|\cdot\Hyp_{X_n}(\gamma_i)\right)\cdot \I(\alpha,\gamma_i)+C(\alpha,\bar\gamma).
\end{equation}
If $\alpha$ is disjoint from all $\gamma_i$, then its length is bounded along $Y_n$ and hence also along $X_n$. In particular $\I(\alpha,(F,\mu))=0$ which means that $(F,\mu)=\sum_{i=1}^k c_i\gamma_i$. 

So let $\alpha$ and $\beta$ be simple closed curves that have positive intersection
number with $(F_\theta,\mu_\theta)$. 

If the lengths ratios of $\gamma_i$ are constant, then from the equation \ref{equal} and the fact that the hyperbolic lengths of any $\gamma_i$ and $\gamma_j$ go to 0 as $n\to\infty$, we have
\begin{equation}
\label{bigtwist}
\left |\Hyp_{X_n}(\alpha)-\sum_{i=1}^k \left(\omega(\Hyp_{X_n}(\gamma_1))+|k_nn_1|\cdot\Hyp_{X_n}(\gamma_1)\right)\cdot \I(\alpha,\gamma_i)\right |\leq C'(\alpha,\bar\gamma).
\end{equation} 
 Then we have 
\begin{equation}
\frac{\Hyp_{X_n}(\alpha)}{\Hyp_{X_n}(\beta)}=\frac{\sum_{i=1}^k \left(\omega(\Hyp_{X_n}(\gamma_1))+|k_nn_1|\cdot\Hyp_{X_n}(\gamma_1)\right)\cdot \I(\alpha,\gamma_i)+O(1)}{\sum_{i=1}^k \left(\omega(\Hyp_{X_n}(\gamma_1))+|k_nn_1|\cdot\Hyp_{X_n}(\gamma_1)\right)\cdot \I(\beta,\gamma_i)+O(1)}\to \frac{\sum_i^k\I(\alpha,\gamma_i)}{\sum_i^k\I(\beta,\gamma_i)},
\end{equation}
that is, $$[(F,\mu)] = [\sum_{j=1}^k \gamma_j].$$

If the lengths ratios for $\gamma_i$'s are bounded,  the quantities $$\omega(\Hyp_{X_n}(\gamma_i))+|k_nn_i|\cdot \Hyp_{X_n}(\gamma_i),\,\, i\in\{1,\ldots, k\}$$  are also
comparable for each $X_n$.  
 Hence, knowing that $X_n$ converges to $[(F,\mu)]$, we conclude that $[(F,\mu)]=[\sum_i c_i\gamma_i]$ where  $c_i>0$ for all $i$.

\item
The last case to consider is that $\Hyp_{Y_n}(\gamma_i)\to +\infty$. We can assume that $Y_n$ converges to some $[(G,\nu)]\in \PMF.$ Note that  the sequence $|k_n|\to +\infty. $ 

By Lemma \ref{intersection_zero},
$ (G,\nu)$ has  intersection number 0 with some $(F_{\theta'},\mu_{\theta'})$ for $\theta'\in [\theta_1,\theta_2]$. Since  the pair of foliations $(F_\theta,\mu_\theta)$ and $(F_{\theta'},\mu_{\theta'})$
fills the surface, we know that  $$\I((G,\nu),(F_\theta,\mu_\theta))>0.$$ Moreover, the assumption that the ratio of hyperbolic lengths of $\gamma_i$ is bounded (or constant) on $D(X,q)$ implies that for all $i\in \{1,\ldots,k\}$ we have $\I(\gamma_i,(G,\nu))>0$. 
This means that along the sequence $Y_n$, for any simple curve $\alpha$
\begin{equation}\label{comparable}\frac{\Hyp_{Y_n}(\alpha) }{\Hyp_{Y_n}(\gamma_i)}\leq O(1).
\end{equation}
 
Morevover, we know that 
$$
|\Hyp_{X_n}(\alpha)- \sum_{i=1}^k k_n n_i \Hyp_{Y_n}(\gamma_i)\I(\alpha,\gamma_i)|\leq \Hyp_{Y_n}(\alpha)+C.
$$
and together with \ref{comparable}, as $n\to \infty$, if $\alpha$ intersects at least some $\gamma_i$, we have
\begin{equation} \label{gammalong}\frac{\Hyp_{X_n}(\alpha)}{k_n\sum_{i=1}^k  n_i \Hyp_{Y_n}(\gamma_i)\I(\alpha,\gamma_i)}\to 1.
\end{equation}
If $\alpha$ is disjoint from  $(F_\theta,\mu_\theta)$, then $\Hyp_{X_n}(\alpha)=\Hyp_{Y_n}(\alpha)$ and so 
$$\frac{\Hyp_{X_n}(\alpha) }{\Hyp_{X_n}(\gamma_i)}\leq O(1)$$ which implies that $\I(\alpha,(F,\mu))=0$. This means that $(F,\mu)=\sum_{i=1}^k c_i\gamma_i$.
If the ratios of lengths of $\gamma_i$ are constant, taking into account  Equation \ref{equal}, this implies that the sequence $X_n$ converges to $[(F,\mu)]=[\sum_{j=1}^k \gamma_j]=[(F_\theta,\mu_\theta)]$. 

For the case when the ratios of hyperbolic lengths of parallel curves are bounded, we still have \ref{gammalong} and \ref{comparable}, so 
 since the quantities $n_i\Hyp_{X_n}(\gamma_i),\, i\in\{1,2,\ldots,k\}$, are comparable for each $n$, the coefficients $c_i$ are all non-zero.

\end{enumerate}
We have shown one direction of the proposition. For the other direction, suppose that for some $i\neq j$, and some point $Y\in D(X,q)$,  
$$\frac{\Hyp_X(\gamma_i)}{\Hyp_X(\gamma_j)}\neq \frac{\Hyp_Y(\gamma_i)}{\Hyp_Y(\gamma_j)}.$$
Then the sequences $X_n:=T_{\bar\gamma}^nX$ and $Y_n:=T_{\bar\gamma}^nY$ converge to different points. This finishes the proof of the proposition. 
\end{proof}

\section{Disks with small limit sets} In this section we prove our main theorem, that is, that   \Teich disks with  limit set in $\PMF$ homeomorphic to a circle, are very rare.  

Suppose $D(X,q)$ is a \Teich disk through $X\in T(S)$ defined by a quadratic differential $q$ that has this property. We first observe that in this case any \Teich ray in $D(X,q)$ based at $X$   converges to the corresponding projectivized measured foliation. 
\begin{lemma} Let $(X,q)$ be a flat surface, let $D(X,q)$ be the corresponding \Teich disk  and suppose that  its limit set $\Lambda(X,q)$ is exactly  $C(X,q)$. Then for any direction $\theta$ on $(X,q)$, the \Teich ray $X_t^\theta=X(e^{i2\theta}q,t)$ converges to the foliation $[(F_\theta,\mu_\theta)]$.

\end{lemma}\label{rays_converge}
 \begin{proof} Let $\theta\in \RR/\pi\ZZ$. By Lemma \ref{intersection_zero} any limit point of the ray $X^\theta_t$ defined by $(F_\theta,\mu_\theta)$  has intersection number zero with $(F_\theta,\mu_\theta)$. On the other hand, for any $\theta'\neq \theta $, the foliations $(F_\theta,\mu_\theta)$ and  $(F_{\theta'},\mu_{\theta'})$ are filling and therefore have 
 strictly positive intersection number. Hence $X_t$ can only have one accumulation point, and that point has to be $[(F_\theta,\mu_\theta)]$.
\end{proof}
  This in fact turns out to be a significant restriction. 
\begin{corollary}\label{square} Let $(X,q)$ be a flat surface such that the limit set $\Lambda(X,q)$ of the \Teich disk is  $C(X,q)$. Then the $SL(2,\mathbb R)$-- orbit of $(X,q)$ 
contains a square-tiled surface. In particular, $(X,q)$ is a Veech surface.  Moreover, for any completely periodic direction $\theta$, the cylinders of the corresponding foliation have equal heights.
\end{corollary} 
\begin{proof}
By a theorem of Masur (\cite{masur:CT}, Theorem 2),  we can choose two transverse directions $\theta_1$ and $\theta_2$ on $(X,q)$ such that the corresponding foliations  have at least one cylinder. By the Proposition \ref{periodic},  these directions are Strebel.

Let $X_t^j$ be the \Teich geodesic ray based at $X$ and defined by $e^{i2\theta_j}q$ for $j=1,2$. Combining Theorem \ref{convrays} of Masur and Lemma \ref{rays_converge} we conclude that the weights of the cylinder curves in $(F_{\theta_j},\mu_j)$
are equal, which means that the cylinders of $(F_{\theta_j},\mu_j)$ have equal heights.

This implies that the flat surface $(X,q)$ is tiled by isometric parallelograms, their sides are segments of saddle connections of slopes $\theta_1$ and $\theta_2$. Hence there is a square-tiled surface in the $SL(2,\mathbb R)$- orbit  $(X,q)$.  In particular $(X,q)$  is a Veech surface. The last statement follows also from the argument in the previous paragraph.

\end{proof}

We will call the property of having in each  periodic direction cylinders of equal heights the \textbf{balanced heights property}.

\begin{theorem}\label{malo} For any $g\geq 2$, there are at most finitely many $SL(2,\mathbb R)$--orbits of square-tiled surfaces with balanced heights property.
\end{theorem}
\begin{proof} Let $(X,q)$ be square-tiled, and suppose it has the balanced heights property.  We will assume that the number $N$ of squares used to tile $(X,q)$ is the smallest possible for the orbit of $(X,q)$, that is, any other square-tilled surface in the $SL(2,\mathbb R)$-orbit of $(X,q)$ needs at least $N$ squares. 
It suffices to prove that  the number $N$ is bounded by some constant that depends only on the genus of the underlying topological surface. 
The argument goes as follows. 

Scale $q$ so that  the squares have area 1. Then minimality of $N$ implies that the vertical and horizontal cylinders are of height 1. We claim that in any other periodic direction, the height of cylinders are at most $1$. Indeed, let $\mathcal C$ be a horizontal cylinder and suppose it has area $M$, that is, it is tiled with $M$ unit squares. Choose $\theta$  any non-horizontal periodic direction. Saddle connections of slope  $\theta$ cut $\mathcal C$ 
into equal parallelograms.  The number of the parallelograms is at least $M$, since we assumed that $(X,q)$ has smallest number of squares. Let $p$ and $q$ be two distinct vertices on the same boundary component of $\mathcal C$ such that the distance $d$ between them along that boundary component is smallest. Then $d$ is at most one, which implies that the height of cylinders with slope $\theta$ is at most $1$.

Now we prove that $N$ is bounded by a constant that depends only on $g$. Let $m$ be the sum of multiplicities of all singularities of $q$. By a theorem of Vorobets (\cite{vorobets:PG}),  $(X,q)$ has a flat cylinder of length 
$\ell\leq 2^{2^{4m}}\sqrt{N}$ and area $A\geq\frac{N}{m}$. That is, there is a cylinder of height  $$h=\frac{A}{\ell}=\frac{\sqrt N}{m2^{2^{4m}}}.$$ On the other hand, we just showed that this height cannot be greater than $1$. Hence
$$N\leq m^22^{2^{4m+1}}.$$

 \end{proof}
We can now prove the main theorem of the paper.
\begin{proof}[Proof of Theorem \ref{small}] The statement follows directly from the Corollary \ref{square} and Theorem \ref{malo}. By Corollary \ref{square}, any flat surface $(X,q)$ of genus $g\geq 2$ for which the limit set $\Lambda(X,q)$ of the corresponding  \Teich disk is  equal to the circle $C(X,q)$ can be assumed square-tiled with  heights property. By Theorem \ref{malo} there are finitely many orbits of such surfaces for any given $g$. 

\end{proof}
 Here is a curious observation. One application for it at the moment is finding flat surfaces with heights property. 
 \begin{lemma}\label{corners} 
Let $(X,q)$ be a square-tiled surface where  no cylinder curve passes through a vertex of a square. Then $(X,q)$ has balanced heights property.
 \end{lemma}
 \begin{proof} Let $(X,q)$ be any square tiled surface of genus $g$ with the property that simple closed geodesics don't pass through corners of the squares. We want to show that $(X,q)$ has  heights property.
We assume that the sides of the tiles are vertical and horizontal and the tiles are unit squares. Then horizontal and vertical cylinders have all height one. Choose any other completely periodic directions and let $\frac{p}{q}$ be its slope, with $p$ and $q$ coprime. To see that the cylinders in this direction have equal heights, it suffices
 to prove that the saddle connections of slope $p/q$ cut every horizontal side of every tile in $p$ segments of length $1/p$.
 
  Identify every horizontal side of a square with $[0,1]$. Consider a straight segment of slope $p/q$ starting at a corner of a tile. It cuts through exactly $p-1$ (counting multiplicities) horizontal sides at points $k/p$ for $k=1\ldots p-1$, before hitting another corner. At the same time, a straight segment of slope $p/q$ starting on some horizontal side at any of these points will hit a corner. By hypothesis it has to be a singularity, and hence the segment is contained in a saddle connection. This finishes the proof.
 \end{proof}

 \section{Veech surfaces} In this section we consider limit sets of  \Teich disks of Veech surfaces.

 The following lemma tells us what are the short curves along the Teichm\"uller disk, and is probably well known. We give a proof for convenience of the reader.
 \begin{lemma} \label{thickthin}
 Let $(X,q)$ be a Veech surface. 
  There is $c>0$ so that the following holds. Let $\epsilon\in (0,\epsilon_M)$. There exists a $\delta>0$ such that, for any $Y\in D(X,q)$,  if there is a saddle connection on $Y$ of flat length at most $\delta$, than the core curves of cylinders  in the direction of a $\delta$-short saddle connection has hyperbolic length less than $\epsilon$ on $Y$ and any other curve has hyperbolic length at least $c$ on $Y$. 
\end{lemma}
\begin{proof} Since the flat surface $(X,q)$ is Veech, by Theorem \ref{SW} due to Smillie and Weiss, there exists $s>0$  such that for any direction on $(X,q)$  the ratio of flat lengths of any two saddle connections in that direction  is at most $s$. Clearly same $s$ works  for any surface $Y$ in the $SL(2,\RR)$ orbit of $(X,q)$. 

Suppose $Y$ has a saddle connection in some direction $\theta$ of length smaller than $\delta$. Then any saddle connection in that direction is of length at most $s\delta$. 
Let $\alpha_i,\, i\subset\{1,\ldots,k\}$ be the core curves of the flat cylinders in direction $\theta$. The flat length of $\alpha_i$ is  at most $(6g-6)s\delta$. Indeed, there are at most $3g-3$ saddle connections in a given direction,  a boundary component of a cylinder is a union of parallel saddle connections, where some saddle connections might appear twice. 

Let $A$ be the smallest  area of a (maximal) flat cylinder on $Y$, note that $A$ depends on $(X,q)$ only. Then the height of the  flat cylinder of $\alpha_i$ is at least $h=\frac{A}{(6g-6)s\delta}$, which implies that  the extremal length  $\Ext_Y(\alpha_i)$ is at most $\frac{((6g-6)s\delta)^2}{A}$. Hence by Theorem \ref{maskit} the hyperbolic length of $\alpha_i$ satisfies  
$$\Hyp_Y(\alpha_i)\leq\frac{((6g-6)s\delta)^2\pi}{A}.$$ 
So taking $\delta=\frac{\sqrt{\epsilon A}}{(6g-6)s\sqrt{\pi}}$ guarantees that if there is a saddle connection of length smaller than $\delta$, then all the cylinder curves in that direction have hyperbolic length at most $\epsilon$. 

Let $\beta$ be any other curve 
on $Y$.  If it intersects some $\alpha_i$, by the Collar lemma we have $$\Hyp_Y(\beta)\geq \omega(\epsilon).$$ Suppose $\beta$ is disjoint from all $\alpha_i$. Then the flat geodesic representative of $\beta$ is a union of saddle connections in the direction $\theta$, and it is contained in a subsurface $W$ that is a connected component of the compliment the flat cylinders in direction  $\theta$. The flat metric area of $W$ is zero. Let $\delta_m\leq \delta$ be the length of the shortest saddle connection in direction $\theta$ on $Y$. To get a lower bound for the extremal length of $\beta$, consider a metric  $\rho$ which coincides with  the flat metric on $Y$ in the $\delta_m$-neighborhood of $W$, and zero elsewhere. Then the length of $\beta$ in the metric $\rho$  is at least $\delta_m$, and the $\rho$-area of $Y$ is at most $\delta_m^2sn$ where $n=n(g)$ is basically the maximal number of boundary components of a connected subsurface of genus $g$ times the maximal number of saddle connections in a given direction on a surface of genus $g$. We then get 
$$\Ext_Y(\beta)\geq \frac{\delta_m^2}{\delta_m^2sn}=\frac{1}{sn}.$$
It follows from  Theorem \ref{maskit}, $\Hyp_Y(\beta)\geq \frac{1}{sne}.$ So we can take $c=\min\{\frac{1}{sne},\omega(\epsilon)\}$.

\end{proof}
 \begin{lemma} \label{boundedratio}Let $(X,q)$ be a Veech surface. Let $\gamma_1$ and $\gamma_2$ be parallel cylinder curves on $(X,q)$. Then the function 
 $$D(X,q):\to (0,+\infty),\,\, Y\mapsto \frac{\Hyp_Y(\gamma_1)}{\Hyp_Y(\gamma_2)}$$
 is bounded.
\end{lemma}
\begin{proof}
First we fix notation for several constants that depend only on the Veech surface $(X,q)$. Scale $q$ so that $X$ has area 1 in the  $q$ metric. Let $A$ be the smallest area of a maximal flat cylinder on any surface in the $SL(2,\RR)$-orbit of $(X,q)$. Similarly, let $L$ be the maximal ratio of flat lengths of parallel cylinder curves, $M$ the maximal ratio of moduli of parallel flat cylinders  and $H$ maximal ratio of heights of parallel flat cylinders.

If $\gamma$ is the core curve of a flat cylinder, than for any $Y\in D(X,q)$ we have by definition
$$\ell^2_Y(\gamma)\leq\Ext_Y(\gamma)\leq\frac{\ell^2_Y(\gamma)}{A}.$$
Then it follows  from Theorem \ref{maskit} that if $\ell_Y(\gamma)$ is bounded above by some $\ell$, then 
$$\frac{2}{e^{\frac{\pi \ell^2}{A}}}\ell_Y^2(\gamma)\leq\Hyp_Y(\gamma)\leq \frac{\pi}{A}\ell_Y^2(\gamma).$$
 So since ratio of flat lengths of $\gamma_1$ and $\gamma_2$ is constant on $D(X,q)$, we conclude that if flat lengths are bounded by $\ell$, the hyperbolic lengths 
 \begin{equation}\label{hyp_short} \frac{2A}{\pi e^{\frac{\pi \ell^2}{A}}}\left (\frac{\ell_X(\gamma_1)}{\ell_X(\gamma_2)}\right )^2\leq\frac{\Hyp_Y(\gamma_1)}{\Hyp_Y(\gamma_2)}\leq \frac{\pi e^{\frac{\pi \ell^2}{A}}}{2A}\left (\frac{\ell_X(\gamma_1)}{\ell_X(\gamma_2)}\right )^2\leq \frac{\pi e^{\frac{\pi \ell^2}{A}}}{2A} L^2.
\end{equation}
We will now focus on the case  when the curves $\gamma_1$ and $\gamma_2$ are long.


Fix a small $\epsilon>0$ and let $c$ and $\delta$ be as in Lemma \ref{thickthin}, reducing $\epsilon$ if needed to have $c>\epsilon$. 
If $Y$ has no saddle connection of length smaller than $\delta$,  then it is in $\epsilon'$-thick part of $T(S)$ where $\epsilon'=\min\{1,2e\delta^2\}$.  Then from Theorem 19 of \cite{rafi:LQC}, for some universal constant $C$ we have
$$\frac{\Hyp_Y(\gamma_1)}{\Hyp_Y(\gamma_2)}\leq C \frac{\ell_Y(\gamma_1)}{\ell_Y(\gamma_2)}\leq CL.$$

Finally, suppose that  $Y$ has a $\delta$-short saddle connection in some direction $\theta$. Denote $(F_\theta,\mu_\theta)=\sum_{i=1}^kh_i\alpha_i$ the foliation in direction $\theta$. Then by Lemma \ref{thickthin},  $Y$ is in $\epsilon$-thin part of the \Teich space and $\Hyp_Y(\alpha_i)\leq \epsilon$ for all $i$, while any other curve has hyperbolic length at least $c$. 

Denote $W$ the thick part of $Y$, and $\mu_W$ a short marking  for $W$ in the hyperbolic metric of $Y$.  Note that $\mu_W$ consists of curves whose flat geodesic representatives are a union of saddle connections in direction $\theta$. Since $(X,q)$ is a Veech surface, the lengths of saddle connections in the same direction have ratio bounded by some $s>0$, where $s$ depends only on $(X,q)$. Hence there is some $D>0$ so that any curve in $\mu_W$ contains at most $D$ copies of any saddle connection. This and the fact that $\gamma_1$ and $\gamma_2$ are parallel curves in a different from $\theta$ direction and hence intersect every component of $W$ at finitely many points imply that for some universal constant $D'$
$$\I(\gamma_i,\mu_W)\leq D' \sum_l \I(\gamma_i,\alpha_l).$$
This implies that most of the hyperbolic length of $\gamma_i$ comes from crossing the thin part of $Y$: up to a universally bounded multiplicative error $\Hyp_{Y}(\gamma_j)$ is
$$ \sum_{i=1}^{k}\I(\gamma_j,\alpha_i)\left(\log\frac{1}{\Hyp_Y(\alpha_i)}+tw_{q_Y}(\alpha_i,\gamma_j)\Hyp_Y(\alpha_i)\right).$$
Here $tw_{q_Y}(\alpha_i,\gamma_j)$ is basically he number of times an arc of $\gamma_j$ twists about $\alpha_i$. It depends only on the modulus of the cylinder of $\alpha_i$ and the angle between $\gamma_j$ and $\alpha_i$. Moduli of flat cylinders in the same direction are equal up to multiplicative error $M$, and angle $\gamma_1$ and $\gamma_2$  make with any $\alpha_i$ they intersect is the same. Also $\frac{\Hyp_Y(\alpha_i)}{\Hyp_Y(\alpha_j)}$ is bounded by \ref{hyp_short}. So we have for some $B_1>0$ and $B_2>0$, assuming that both $\gamma_1$ and $\gamma_2$ intersect $\alpha_1$, that
$$\Hyp_Y(\gamma_1)\leq B_1 \left(\log\frac{1}{\Hyp_Y(\alpha_1)}+tw_{q_Y}(\alpha_1,\gamma_1)\Hyp_Y(\alpha_1)\right)\sum_{i=1}^{k}\I(\gamma_1,\alpha_i)$$
and 
$$\Hyp_Y(\gamma_2)\geq B_2\left(\log\frac{1}{\Hyp_Y(\alpha_1)}+tw_{q_Y}(\alpha_1,\gamma_1)\Hyp_Y(\alpha_1)\right)\sum_{i=1}^{k}\I(\gamma_2,\alpha_i).$$
Finally
$$\frac{\Hyp_Y(\gamma_1)}{\Hyp_Y(\gamma_2)}\leq \frac{B_1}{B_2}\frac{\sum_{i=1}^{k}\I(\gamma_1,\alpha_i)}{\sum_{i=1}^{k}\I(\gamma_2,\alpha_i)}\leq H\frac{B_1}{B_2}\frac{\I(\gamma_1,(F_\theta,\mu_\theta))}{\I(\gamma_2,(F_\theta,\mu_\theta))}.$$ This finishes the proof of the Lemma.
\end{proof}

 Let $(X,q)$ be a Veech surface. Recall the notation 
 $$M(X,q)=C(X,q)\cup\left(\underset{\theta\in \RR/2\pi\ZZ}{\cup}\Delta^{\mathrm{o}}(F_\theta,\mu_\theta)\right).$$

\begin{theorem}[Theorem \ref{Veech}] Let $(X,q)$ be a Veech surface. Then 
    $$\Lambda(X,q)\subset M(X,q).$$
\end{theorem}
\begin{proof} 

Let $X_n$ be a sequence in $D(X,q)$, and suppose that it converges in $\mathcal{PMF}$ to some $[(F,\mu)]$.
By Lemma \ref{intersection_zero}, there is a direction $\theta$ such that the foliation $(F_\theta,\mu_\theta)$  has  intersection number zero
with $(F,\mu)$.

 If $\theta$ is a uniquely ergodic direction, then $\I((F,\mu),(F_\theta,\mu_\theta))=0$ implies that $$(F,\mu)=(F_\theta,s\cdot\nu_\theta)$$ for some $s>0$, nothing to
prove in this case. 

Suppose that  $\theta$ is not uniquely ergodic. Then by Veech dichotomy $\theta$ is a completely periodic direction and hence 
$(F_\theta,\mu_\theta)=\sum_{i=1}^k c_i\gamma_i$, where $\gamma_i$ are core curves of parallel cylinders of heights $c_i$. Note that by Lemma \ref{boundedratio}
for any $i,j\in \{1,\ldots, k\}$ the ratio $\frac{\Hyp_{Y}(\gamma_i)}{\Hyp_Y(\gamma_j)}$ is bounded on $D(X,q)$. Then by Proposition \ref{onthenose}, 
$$(F,\mu)=\sum_{i=1}^k h_i\gamma_i$$
where all the coefficients $h_i$ are strictly positive. This finishes the proof of the theorem.

\end{proof}
\begin{corollary}\label{small}Let $(X,q)$ be a Veech surface. The limit set of $D(X,q)$ in $\PMF$ is equal to $C(X,q)$ if and only if  for any two parallel curves $\gamma_1$ and $\gamma_2$ on $(X,q)$, the ratio
of hyperbolic lengths of $\gamma_1$ and $\gamma_2$ is a constant function on $D(X,q)$.
\end{corollary} 
 \section{Examples} 
 In this section we give examples of \Teich disks with small limit sets.
 Recall that a square-tiled surface $(X,q)$ is a branched covering of a square torus $T$ branched over one point. Denote $p:X\to T$ the branched cover, and let $0\in T$ be the branch point. 
 Denote $X^*=X\setminus \{p^{-1}(0)\}$ and $T^*=T\setminus\{0\}$. A square-tiled surface is
 called normal if $p:X^*\to T^*$ is a normal covering.
 \begin{theorem} If $(X,q)$ is a normal square-tiled surface, then  $\Lambda(X,q)=C(X,q)$.
 \end{theorem}

 \begin{proof} 
 
 Let $G$ be the group of deck transformations of the covering $p:X^*\to T^*$. The differential $q$ is a pullback of a unique quadratic differential on $T^*$. Any $g\in G$ is an automorphism of the Riemann surface $X^*$, and can be extended to the punctures. Hence $G$ acts on the complex structure of $X$ by automorphisms and preserves $q$. Then the hyperbolic metric of $X$ is also preserved. Also the fact that $G$ acts transitively on the fibers of the covering $p:X^*\to T^* $  implies that it acts transitively on the set of  the core curves of the cylinders in any completely periodic direction. More precisely, fix a completely periodic direction on $X$ and let $\gamma_1,\ldots,\gamma_k$ be the cylinder curves.  For any $i\in \{1,\ldots,k\}$ there is an element $g\in G$ that send $\gamma_1$ to $\gamma_k$. But this implies that the curves $\gamma_1,\ldots,\gamma_k$ have the same hyperbolic length on $X$. Taking different translation structures on the punctured torus one obtains any point in $D(X,q)$ via the covering map $p$, and the group $G$ acts on any $Y\in D(X)$ by hyperbolic isometries. We conclude that any two cylinder curves in the same  periodic direction have equal hyperbolic lengths throughout the \Teich disk.  By Corollary \ref{small}, $\Lambda(X,q)=C(X,q)$.
   \end{proof}
 Most famous example of a normal origami is \cite{Herr_EXO}, also see \cite{Schmith:HU} with many more examples in genus $g$ odd or $g=1[3]$.   
  \bibliographystyle{alpha}
  \bibliography{main}
\end{document}